\documentclass[11pt]{amsart}
\usepackage{amsmath,amsthm,amsfonts,amssymb,latexsym,mathrsfs,graphicx}

\usepackage{hyperref}
\usepackage{longtable}

\usepackage{float}

\usepackage{comment}

\usepackage{enumerate}
\usepackage[shortlabels]{enumitem}
\usepackage{color}

\usepackage[style=alphabetic, sorting=nyt, backend=biber]{biblatex}
\newcommand{\say}[1]{``#1''}

\newcommand{\Z}{\mathbb Z}
\newcommand{\Q}{\mathbb Q}

\newcommand{\la}{\langle}
\newcommand{\ra}{\rangle}

\newtheorem{theorem}{Theorem}[section]

\newtheorem{lemma}[theorem]{Lemma}
\newtheorem{proposition}[theorem]{Proposition}
\theoremstyle{definition}
\newtheorem{definition}[theorem]{Definition}

\newtheorem{rem}[theorem]{Remark}

\newtheorem*{ThmA}{Theorem A}
\newtheorem*{ThmB}{Theorem B}
\newtheorem*{ThmC}{Theorem C}

\def\sg#1{\mathcal{S}_#1}
\def\rg#1{\mathcal{R}_#1}

\def\irr#1{{\rm Irr}(#1)}
\def\cent#1#2{{\bf C}_{#1}(#2)}

\def\irr#1{{\rm Irr}(#1)}

\def\cent#1#2{{\bf C}_{#1}(#2)}

\def\Z{{\Bbb Z}}

\def\Q{{\Bbb Q}}
\def\irr#1{{\rm Irr}(#1)}

\def\cent#1#2{{\bf C}_{#1}(#2)}

\mathchardef\coso="2023

\begin{document}

\title{Uniformly semi-rational simple groups}

\begin{abstract}
    A finite group $G$ is called \emph{uniformly semi-rational} if there exists an integer $r$ such that the generators of every cyclic sugroup $\la x \ra$ of $G$ lie in at most two conjugacy classes, namely $x^G$ or $(x^r)^G$. In this paper, we provide a classification of uniformly semi-rational non-abelian simple groups with particular focus on alternating groups. 
\end{abstract}

\author[]{Marco Vergani}
\address{Marco Vergani, Dipartimento di Matematica e Informatica U. Dini,\newline
Universit\`a degli Studi di Firenze, viale Morgagni 67/a,
50134 Firenze, Italy.}
\email{marco.vergani@unifi.it}

\thanks{The author is partially supported by INdAM-GNSAGA. This research is also funded by the European Union-Next Generation EU, Missione 4 Componente 1, CUP B53D23009410006, PRIN 2022 2022PSTWLB - Group Theory and Applications.}


\maketitle

\section{Introduction}


Let $G$ be a finite group and $x$ an element of $G$. Given an integer $r_x$, we say that $x$ is \emph{$r_x$-semi-rational} in $G$ if every generator of the cyclic subgroup $\la x \ra$ is conjugate either to $x$ or to $x^{r_x}$ (in this case, when there is no need to emphasize the role of $r_x$, we simply say that $x$ is a \emph{semi-rational} element of $G$), and the group $G$ is \emph{semi-rational} if all its elements are semi-rational. This is a generalization of the well-known concept of a \emph{rational} group (see \cite{Kletzing1984,FeitSeitz1989,Thompson2008,IsaacNavarro2012}). 

In this paper, we investigate a strengthening of the property of being semi-rational. Namely, we define a finite group to be \emph{$r$-semi-rational} if every element of the group is $r$-semi-rational \emph{for the same integer $r$}, i.e., if we have $r_x=r$ for every $x$ in $G$. If such an $r$ exists, sometimes we simply say that the group is \emph{uniformly semi-rational}. We note that the well-known class of the \emph{cut} groups or  \emph{inverse semi-rational} groups (see  \parencite{CD}) is included in the class of uniformly semi-rational groups, since they are precisely the $-1$-semi-rational groups.
  
Our aim is to study non-solvable uniformly semi-rational groups and, in particular, we will focus on finite non-abelian simple groups. In 1988, Feit and Seitz studied the composition factors of rational groups and determined the non-abelian simple groups that are rational: there are only two of them, which are $PSp_6(2)$ and $P\Omega_8^+(2)$. In 2017, Trefethen \cite{CutTrefethen2017} considered other generalizations of rational groups, namely, quadratic rational groups (see Section~2)  and $m$-rational groups, providing a complete classification of the non-abelian simple groups that can appear as composition factors of groups with these properties. Also, in \cite{Alavi2017} Alavi and Daneshkhah classified the semi-rational simple groups; it is easily seen that every alternating group is semi-rational, but it turns out that there are only finitely many non-alternating simple groups which are semi-rational. 

Since every uniformly semi-rational group is semi-rational, it is only a matter of computation to determine which non-alternating simple groups are uniformly semi-rational (see Theorem~C). 
On the other hand, the problem of determining which alternating groups are uniformly semi-rational requires an entirely different approach, and it is the main topic of this paper (we mention that the inverse semi-rational alternating group have been studied in  \cite{Ferraz}). After developing some tools to study the rationality of permutations in an alternating group, we will prove the following:

\begin{ThmA} The alternating group $A_n$ is uniformly semi-rational if and only if $n<23$ and $n\not\in\{16,21\}$.
\end{ThmA}

In our proof of Theorem~A, a crucial step will be to prove an arithmetical result which may be interesting on its own: 

\begin{ThmB}   Let $n$ be a positive integer. Then $n$ is the sum of pairwise distinct odd positive integers whose product is a square if and only if  $n\in\{9,10\}$, or $n\geq 23$ and $n\not \in \{27, 28, 29, 36, 37,
  38, 42, 43, 44, 45, 46, 52, 53, 54, 62, 67, 68\}$. 
  \end{ThmB}
  
As mentioned above, in the last section of this paper we deal with the other finite non-abelian simple groups:

\begin{ThmC}  Let $G$ be a finite non-abelian simple group which is not an alternating group. If $G$ is semi-rational, then $G$ is uniformly semi-rational.  
\end{ThmC}

In the tables appearing in the paper, we will also provide the integers~$r$ for which a uniformly semi-rational simple group is $r$-semi-rational. 

Every group considered in this paper is assumed to be a finite group.

\section{Uniformly semi-rational alternating groups}
Let $G$ be a group, and let $\irr G$ denote the set of the irreducible complex characters $G$. For $g\in G$, we define $\Q(g)=\Q(\chi(g)\mid\chi\in\irr G)$ to be the field generated over $\Q$ by the values in the column of the character table of $G$ corresponding to $g$; dually, for $\chi\in\irr G$, we set $\Q(\chi)=\Q(\chi(g)\mid g\in G)$ (the field generated over $\Q$ by the values in the row of the character table corresponding to $\chi$). It turns out that $G$ is semi-rational if and only if, for every $g\in G$, the field $\Q(g)$ is an extension of $\Q$ of degree at most $2$. On the other hand, a group $G$ is said to be \emph{quadratic rational} if, for every $\chi\in\irr G$, we have $|\Q(\chi):\Q|\leq 2$.
 
Alternating groups are both semi-rational and quadratic rational, but only a finite number of them are inverse semi-rational, as seen in \cite{bacjes}.
In this paper we will develop some techniques to explicitly determine the semi-rationality of an element $\pi\in A_n$, namely, the integers $r$ coprime with the exponent of $A_n$ such that $\pi$ is $r$-semi-rational. In particular, we will be able to determine whenever the whole group is $r$-semi-rational or to prove that it is not uniformly semi-rational.


In order to do so, we start by introducing some tools. For $\pi\in A_n$ we write $\pi^{S_n}$ to denote the conjugacy class of $\pi$ in $S_n$, whereas $\pi^{A_n}$ will be the conjugacy class of $\pi$ in $A_n$. Since $A_n$ is a subgroup of index $2$ of $S_n$, either $\pi^{S_n}=\pi^{A_n}$ or $\pi^{S_n}$ splits into two conjugacy classes of $A_n$ having the same size $|\pi^{S_n}|/2$. We also denote with $\alpha(\pi)$ the type of the permutation $\pi$, which is a partition of $n$. The following lemma is well known.

\begin{lemma}
Let $\pi\in A_n$ and consider the conjugacy class $\pi^{S_n}$. Then $\pi^{S_n}$ splits into two conjugacy classes of $A_n$ if and only if the numbers in the partition $\alpha(\pi)$ are pairwise distinct and odd. 
\end{lemma}

For our purposes, it is convenient to identify $S_n$ with ${\rm Sym}(\Omega_n)$ where $\Omega_n=\{0, \dots, n-1\}$. 
Let us consider the cycle $\pi=(0, 1, \dots, m-1)$ where $m\leq n$, and an integer $r$ coprime with $m$: we want to determine an element $\xi\in S_n$ such that $\pi^{\xi}=\xi^{-1}\pi \xi=\pi^r$. 
Observe that we have $$\pi^{\xi}=(\xi(0), \xi(1), \dots, \xi(i), \dots, \xi(m-1))\quad \text{ and }\quad \pi^r=(0, r, 2r, \dots, ir, \dots, (m-1)r)$$ where, slighly abusing the notation, $ir$ stands for the unique number in $\Omega_m$ that is congruent to $ir$ modulo $m$  (for every $i\in\Omega_m$).

Since we want $\pi^\xi=\pi^r$, we can choose $\xi$ to be the permutation that maps $i$ to $ir$ for every $i\in \Omega_m$ and that fixes all the other elements of $\Omega_n$.
We call this permutation the \emph{standard $r$-conjugator} of $\pi$ and we denote it by $\xi^{(r)}_\pi$; we will be interested in understanding if $\xi^{(r)}_{\pi}$ is in $A_n$, in other words, in determining its sign.
Note that the cycle decomposition of $\xi^{(r)}_\pi$ reflects the partition of the set $\Omega_m$ into orbits for the multiplication action of the cyclic subgroup $\la r\ra\subseteq \mathcal{U}(\Z/m\Z)$. 

We also observe that the definitition of the standard $r$-conjugator can be extended in an obvious way to any cycle $\rho=(x_0, \dots, x_{m-1})$ (namely, $\xi_\rho^{(r)}$ is the permutation that maps $x_i$ to $x_{ir}$ for every $i\in\Omega_m$ and fixes every element in $\Omega_n\setminus\{x_0,\ldots,x_{m-1}\}$); it is easily seen that $\xi_\rho^{(r)}$ has the same type as $\xi^{(r)}_\pi$. Clearly, the previous construction can be also extended to any permutation $\sigma$ of $S_n$. In fact, considering a permutation $\sigma=\prod_{i=1}^{l_\sigma}\rho_i$ where $\rho_i$ are disjoint cycles, and an integer $r$ that is coprime with the length of every $\rho_i$, we define the standard $r$-conjugator of $\sigma$ to be $\xi_{\sigma}^{(r)}=\prod_{i=1}^{l_\sigma}\xi_{\rho_i}^{(r)}$. 



\begin{definition}
    Let $r,m$ and $d$ be integers such that $d$ divides $m$ and $r$ is coprime with $m$. We set $o_d(r)$ to the minimum positive integer $t$ such that $r^t\equiv 1$ modulo $d$.
\end{definition}

Observe that $o_d(r)$ is also the minimum positive integer $t$ such that $\frac{m}{d} r^t\equiv \frac{m}{d}$ mod $m$. In what follows, 
we denote by $n_p$ the $p$-part of $n$ (i.e., the largest $p$-power that divides $n$) and by $n_{p'}$ the $p'$-part $\frac{n}{n_p}$ of $n$.


\begin{lemma}\label{lemma3.2}
Let $p$ be an odd prime, $q=p^k$ for a positive integer $k$, and $x\in (\Z/q\Z)^\times$. Then we have $$o_{p^h}(x)_{2}=o(x)_{2}$$ for every integer $h$ such that $1\leq h\leq k$.
\end{lemma}
\begin{proof}
Let us write $x=x_2x_{2'}$ as the product of two elements whose orders are respectively $o(x)_2$ and $o(x)_{2'}$. Also, write $p-1=2^lm$ with $m$ odd. Observe that we have $x^{p^{k-1}m}=x_2^{p^{k-1}m}$, and the order of this element coincides with $o(x)_2$ (in fact, ${p^{k-1}m}$ is the $2'$ part of $|(\Z/q\Z)^\times|$).  
Let us write $x=a+bp^h$, where $a=a_0+a_1p+ \dots + a_{h-1}p^{h-1}$, $a_0\in \{1, 2, \dots, p-1\}$, $a_i\in \{0, \dots, p-1\}$ for $0<i<h$, and $b\in \Z/p^{k}\Z$. Then $x^{p^{k-1}m}= (a+bp^h)^{p^{k-1}m}=a^{p^{k-1}m}$ and, in particular, we have $o(x)_2=o(a)_2$ and $o_{p^h}(x)_2=o_{p^h}(a)_2$. Setting $t=o_{p^h}(a)$, clearly $a^t$ is congruent to $1$ modulo $p$, so $a^t=1+p\alpha$ for a suitable integer $\alpha$. Thus, $a^{tp^{k-1}m}= (a^t)^{p^{k-1}m}=(1+p\alpha)^{p^{k-1}m}\equiv 1+ p^{k-1}mp\alpha\equiv 1$ modulo $q$, and the order of $a$ divides $o_{p^h}(a)p^{k-1}m$. Since $o_{p^h}(a)$ is a divisor of $o(a)$, we conclude that $o_{p^h}(x)_{2}=o_{p^h}(a)_{2}=o(a)_{2}=o(x)_2$, as claimed.
\end{proof}

Now we are able to determine the sign (that we regard here as an element of $\Z/2\Z$) of the standard $r$-conjugator. In what follows, $\varphi$ denotes the Euler function.

\begin{lemma}\label{lemma2}
Let m be an odd positive integer with $m\leq n$, let $\pi$ be the cycle $(0,\ldots,m-1)\in S_n$ and let $r$ be an integer coprime with $m$. Consider the following element of $\Z/2\Z$: $$\tau_m^{(r)}=\sum_{d\mid m, d\neq 1}\frac{\varphi(d)}{o_d(r)}\textnormal{ mod }2.$$
Then the sign of $\xi_\pi^{(r)}$ is $\tau_m^{(r)}$.
\end{lemma}
\begin{proof}
For a divisor $d> 1$ of $m$, consider the elements of additive order $d$ in $\Z/m\Z$: they are represented by integers of the form $\frac{m}{d}k$ with $k$ coprime with $d$, in particular, there are $\varphi(d)$ elements of this kind. Since $r$ is invertible modulo $m$, every $\frac{m}{d}kr^i$ has the same additive order. Thus the cycles $(\frac{m}{d}k, \frac{m}{d}kr, \dots, \frac{m}{d}kr^{o_d(r)-1})$ have all the same size and partition the elements of this kind. Now, we have $\frac{\varphi(d)}{o_d(r)}$ distinct cycles and each cycle has sign $o_d(r)+1$ (mod $2$). Therefore we obtain $$\frac{\varphi(d)}{o_d(r)}(o_d(r) +1)=\varphi(d)+ \frac{\varphi(d)}{o_d(r)}\equiv \frac{\varphi(d)}{o_d(r)}\text{ mod }2$$ since $d\neq 1$ and $\varphi(d)$ is even.
As $\xi_\pi^{(r)}$ is the product of these cycles, the desired conclusion follows.
\end{proof}
We can further improve the previous result.
\begin{lemma}\label{prop1}
Let m be an odd positive integer with $m\leq n$, let $\pi$ be the cycle $(0,\ldots,m-~1)\in S_n$ and let $r$ be an integer coprime with $m$.  Then, considering the prime decomposition $m=\prod_{i=1}^{t} p_i^{k_{p_i}}$ and setting $\mathcal{P}=\{p_1, \dots, p_t\}$, we have
$$sign\left(\xi^{(r)}_\pi\right)=  \sum_{p\in \mathcal{P}} k_p\tau_p^{(r)} ,$$
which is the congruence class of $\sum_{p\in \mathcal{P}}k_p\cdot \frac{(p-1)_2}{o_p(r)_2}$ modulo $2$.
\end{lemma}
\begin{proof}
Suppose $m$ is a prime power, namely $m=p^{k_p}$. For every $1\leq h\leq k_p$ we have that $\frac{\varphi(p^h)}{o_{p^h}(r)} $ is congruent to $ \frac{(p-1)_2}{o_{p^h}(r)_{2}}$ modulo $2$ and, by Lemma~\ref{lemma3.2}, the latter is equal to $\frac{(p-1)_2}{o(r)_{2}}$. Now we apply Lemma~\ref{lemma2}, obtaining that the sign of $\xi^{(r)}_\pi$ is $\tau_m^{(r)}=k_p\cdot \frac{(p-1)_2}{o_p(r)_2}\;({\textnormal{mod }}2)=k_p\cdot \tau_p^{(r)}$ .

Next, let us consider the case where $m=\prod p^{k_p}$ is not a prime power. We want to prove that, for any divisor $d$ of $m$ which is not a prime power, we have $\frac{\varphi(d)}{o_d(r)}\equiv 0$ mod $2$. 

Recall first that $\frac{\varphi(d)}{o_d(r)}$ is congruent to $ \frac{\varphi(d)_2}{o_d(r)_2}$ modulo $2$. Now, 
consider the largest value $v$ in the set $\{(q-1)_2\mid q {\textnormal{ is a prime divisor of }}d\}$, and take a prime $p$ such that $(p-1)_2=v$; also, set $t=d_{p'}$. 
Taking into account that $\varphi(d)_2=(p-1)_2\cdot \varphi(t)_2$, we claim that $o_d(r)_2$ divides $(p-1)_2$. Since the largest $2$-power that divides the exponent of $\mathcal{U}(\Z/d\Z)$ is equal to $v$, for every $r$ in $\mathcal{U}(\Z/d\Z)$ we have that $o_d(r)_2$ divides $v=(p-1)_2$. As $\varphi(t)$ is even, we get $\frac{\varphi(d)_2}{o_d(r)_2}\equiv \frac{(p-1)_2}{o_d(r)_2}\varphi(t)_2\equiv 0$, as wanted.
The desired conclusion easily follows.
\end{proof}

\begin{rem}\label{rem1}
    Let $\pi$ be an element of $A_n$ and $r$ an integer coprime with the exponent of $A_n$. By definition, if $\xi_\pi^{(r)}\in A_n$ then $\pi$ is conjugate in $A_n$ to its $r$th power. However, in the case when the conjugacy class $\pi^{S_n}$ splits in $A_n$, we claim that also the converse is true. 
    In fact, let $\sigma\in A_n$ such that $\pi^\sigma= \pi^r$: then $\pi^\sigma=\pi^{\xi_\pi^{(r)}}$ and $\pi^{\xi_\pi^{(r)}\sigma^{-1}}=\pi$.
Since the conjugacy class of $\pi$ splits in $A_n$, we have $\cent{S_n}\pi=\cent{A_n}\pi$. In particular, $\xi_\pi^{(r)}\sigma^{-1}$ lies in $A_n$ and therefore $\xi_\pi^{(r)}\in A_n$, as claimed.


\end{rem}


Let us now begin the discussion concerning the uniform semi-rationality of $A_n$. Let $\pi$ be a permutation of $A_n$, thus either $\pi^{S_n}=\pi^{A_n}$ or $\pi^{S_n}$ splits in two conjugacy classes of $A_n$. Since $\pi$ is a rational element of $S_n$ (i.e. for every integer $r$ coprime with the order of $\pi$ we have $\pi^r\in\pi^{S_n}$), in the former case $\pi$ is rational in $A_n$ as well. Consider the case when $\pi^{S_n}$ splits, which happens if and only if the type of $\pi$ is a partition of $n$ consisting of odd and pairwise distinct numbers: then $\pi$ is either semi-rational or rational. This last situation will be ``critical" for our discussion and is the focus of the following definition.

\begin{definition}
    Let $n$ be a positive integer. We say that a partition $\lambda$ of $n$ is \emph{odd-distinct} if $\lambda=[\lambda_1, \dots, \lambda_m]$ is such that all the $\lambda_i$ are odd and, for $i\neq j$, we have $\lambda_i\neq\lambda_j$; any permutation in $A_n$ whose type is an odd-distinct partition will be called an odd-distinct permutation. Moreover, an odd-distinct partition $\lambda$ will be called \emph{rational} if the product $\prod_{i=1}^m \lambda_i$ is a square.  

\end{definition}

The previous definition helps us to characterize the rational elements of an alternating group. We already know that every element $\pi$ of $A_n$ whose conjugacy class $\pi^{S_n}$ does not split is rational in $A_n$. On the other hand, we will prove that the odd-distinct permutations that are rational in $A_n$ are exactly the permutations whose type is a rational partition.

\begin{proposition}
    Let $\pi\in A_n$ an odd-distinct permutation. Then the following are equivalent:
    \begin{enumerate}
        \item $\pi$ is rational. 
        \item $\alpha(\pi)$ is a rational partition.
    \end{enumerate}
\end{proposition}
\begin{proof}
    Let $\pi$ be a permutation satisfying (2), i.e. its type $\alpha(\pi)=[\lambda_1, \dots, \lambda_{l}]$ is a rational partition. We claim that the sign of $\xi_\pi^{(r)}$ is $0$ for every integer $r$ coprime with the exponent of $A_n$; by Remark~\ref{rem1}, this will imply (1). Let us consider the prime decomposition of $\lambda_i=\prod_{p\mid \lambda_i}p^{k_{i,p}}$ and denote with $\mathcal{P}$ the set of all primes that divide at least one of the  $\lambda_i$ .
    Applying Lemma~\ref{prop1}, the sign of $\xi_\pi^{(r)}$ is the class of $\sum_{\lambda_i}\sum_{p\mid \lambda_i}k_{i,p}\frac{(p-1)_2}{o_p(r)_2}\equiv \sum_{p\in \mathcal{P}} \frac{(p-1)_2}{o_p(r)_2}  \sum_{\lambda_i}k_{i,p}$ modulo $2$. Since $\alpha(\pi)$ is a rational partition, we have that $\sum_{\lambda_i}k_{i,p}$ is even, thus we proved the claim. 
    
    Now, let us assume that $\pi$ is an odd-distinct rational permutation. As we have observed in Remark \ref{rem1}, $\pi$ is conjugate to its $r$th power if and only if $\xi_\pi^{(r)}$ lies in $A_n$. Therefore we have that the sign of $\xi_m^{(r)}$ is $0$ for every $r$ that is coprime with the order of $\pi$, and in particular we can assume that $r$ is coprime with the exponent of the whole group. Fix a prime $p$ in $\mathcal{P}$ and consider an integer $r_p$ that satisfies $o_p(r_p)_2=(p-1)_2$ and $o_t(r_p)=1$ for any other prime $t\in \mathcal{P}\setminus \{p\}$. The sign of $\xi_\pi^{(r_p)}$ turns out to be congruent to $\sum_{\lambda_i}k_{p,i}$ mod $2$. Since $\pi$ is an odd-distinct rational permutation, we get that the latter number is even, as we desired. 
\end{proof}

Rational partitions are quite \say{rare} whenever $n$ is \say{small}. However, in the following theorem (which is Theorem~B of the Introduction rephrased) we will provide a rational partition for every $n$ which is ``big enough"; for later use, it is convenient to note that such a partition can be always chosen not containing the number $41$. 
\begin{ThmB}\label{partitions}
    Let $n$ be a positive integer, and set $$S=\{27, 28, 29, 36, 37,
  38, 42, 43, 44, 45, 46, 52, 53, 54, 62, 67, 68\}.$$ Then there exists a rational partition of $n$ if and only if $n\in\{9,10\}$, or $n\geq 23$ and $n\not \in S $. 
\end{ThmB}
\begin{proof}
For the ``if" part of the theorem, we will go through a case-by-case analysis that will provide a rational partition for every $n$ such that $n\in\{9,10\}$, or $n\geq 23$ and $n\not\in S$.

   \begin{itemize}
 
 \item $n$ is congruent to $0$, $1$, $25$ or $26$ mod $32$.
        
    \noindent Assume first $n=32k$ where $k$ is a positive integer. 
        \begin{enumerate}
            \item If $k$ is odd, then we can choose the partition $$n\vdash [k, 3k, 7k, 21k]$$
            or 
            $$n\vdash [3k, 5k, 9k, 15k]$$
            \item If $k$ is even then we set $k'=k-1$, so $n=32+32k'$. In this case, assuming $k'\not\in\{1,3\}$, we can choose $$n\vdash [3,5,9,15, 3k',5k',9k',15k']. $$
            If $k'=1$, then choose $$n=32+32=64\vdash [5,11,15,33],$$
            whereas, if $k'=3$, $$n=32+32\cdot 3=128\vdash [5,11, 35, 77].$$
 It will be useful to consider also the alternative $$n=32+32\cdot5=192\vdash[ 3, 5, 15, 169 ]$$ for $k'=5$.           
        \end{enumerate}
Now, the cases when $n$ is congruent to $1$, $25$ or $26$ mod $32$ can be treated adding respectively $1$, $25$ or both these numbers to the partitions above, using the relevant alternatives when needed.    
    
    \medskip
    \item $n$ is congruent to $2$ or $3$ mod $32$.
     
   \noindent Let us start from the case $n\equiv 2$ (mod $32$). Since $n\neq 2$, we have $n=34+32k$ where $k\geq 0$. Observe that $34\vdash[9, 25]$ is a rational partition. Moreover:
    \begin{enumerate}
        \item If $k$ is odd then, for $k\not\in\{3,9,25,41\}$, choose $$n=34+32k\vdash[9, 25, k, 3k, 7k, 21k].$$
        For $k\in\{9,25,41\}$ we can consider $$n=34+32k\vdash[9, 25, 3k, 5k, 9k, 15k].$$
        Finally, choose $n=34+32\cdot 3=130\vdash [3,5,9,25,33,55]$ for $k=3$.
        \item If $k>0$ is even, then we have $n=66+32k'$ for $k'=k-1$.
        If $k'=1$ then take $$n=98\vdash [ 5, 9, 11, 15, 25, 33 ].$$ Then consider $n=34+64+32d=98+32d$ where $d>0$ odd; we can choose the partition of $n$ obtained by concatenating $98\vdash[5,9,11,15,25,33]$ and one among $32d\vdash[d,3d,7d,21d]$ or $[3d,5d,9d,15d]$, except for $d\in\{3,5,11\}$. For the remaining cases, we take $194\vdash[25,169]$, $258\vdash[5,11,15,33,25,169]$, $450\vdash[45,405]$.
    \end{enumerate}
    
    Assume now that $n$ is congruent to $3$ mod $32$. Then we can add $1$ to every partition considered in the previous case, with the exception of $n=131$ for which we can take $131\vdash[1,5,125]$. 
    
    \medskip
    \item $n$ is congruent to $4$, $5$, $29$ or $30$ mod $32$.
     
   \noindent Let us start assuming $n\equiv 4$ (mod $32$). Since it can be checked that $4, 36, 68$ do not admit a rational partition, we assume $n=100+32k$. We can pick the following partitions: $$n=100\vdash[3,7,27,63],$$ $$n= 132\vdash [49, 5, 13, 5\cdot 13],$$ $$n=164\vdash [ 81, 65, 13, 5 ].$$
         Let now $n=164+32h$.
        \begin{enumerate}
            \item If $h$ is odd, then the partition $$n\vdash[5, 13, 65, 81,3h, 5h, 9h, 15h]$$ is rational if $h\not \in \{ 1,9,13,27\}$. Otherwise we can consider $$n\vdash [11, 35, 55, 63, 3h, 5h, 9h, 15h].$$
                
            \item Let $h>0$ be even. Then set $h'=h-1$, so $n=196+32h'$ with $h'$ odd.
            Since $196\vdash [15, 27, 55, 99]$, we have
            $$n\vdash [15, 27, 55, 99, 3h', 5h', 9h', 15h'],$$
            which is rational provided $h'\not \in \{1,3, 5,9, 11,33\}$.
            Otherwise, we pick
            $$n\vdash [7,35, 55, 99 , 3h', 5h', 9h', 15h'].$$
        
 \end{enumerate}
 
Observe that $1,25$ do not appear in any partition presented before unless for $n\in\{324, 356\}$ (which we cover by $[3,35,55,3\cdot 7\cdot 11]$ and $[5,65,9\cdot 13,13\cdot 13]$ respectively), hence our claim is proved also for  $n\equiv 5,29, 30$ (mod $32$) by adding $1$,  $25$ or both, except that the values $n\in\{30, 61, 69, 93, 94\}$ still have to be covered. For these, we pick 
$$30\vdash[3, 27], \quad 61\vdash [1,3,7,15, 35 ],  \quad 69\vdash [ 3, 7, 9, 15, 35 ],$$ $$93\vdash [ 3, 5, 7, 15, 63 ],\quad  94\vdash [ 1,3,5,7,15,63].$$
           
\medskip           
            \item $n$ is congruent to $6$, $7$ or $31$ mod $32$.
            
            \noindent Assume first $n\equiv 6$ (mod $32$), and observe that $70$, $102$ are covered by $[7, 63]$ and $[27, 75]$ respectively. Note also that all the relevant numbers are either of the form $70+32k$ or of the form $102+32k$, for a suitable odd number $k$.
            \begin{enumerate}
                \item If $n=70+32k$, then we have $$n\vdash [7,63, 3k, 5k, 9k, 15k],$$ and this is a rational partition unless $k\in\{7,21\}$. In those cases consider $$294\vdash[ 3, 5, 7, 21, 43, 5\cdot 43 ],\quad 742\vdash[ 25\cdot 7, 81\cdot 7 ].$$
                Since we want to add $25$ for the case $n\equiv 31$ (mod $32$), it will be convenient to cover the case $k=5$ also with the alternative partition
                $230\vdash[ 23, 23\cdot 9 ].$
                \item If $n=102+32k$, then we have $$n\vdash [27, 75, 3k, 5k, 9k, 15k]$$ which is rational unless  $k\in\{3,5,9,15,25\}$;  for $k\in\{3,5,15\}$ we pick $$n\vdash[27, 75, k, 3k, 7k, 21k],$$ whereas  $$102+32\cdot 9=390\vdash [15, 3\cdot 125]]$$ and $$102+32\cdot 25=902\vdash [11,81\cdot 11]$$
cover the remaining cases. Adding $1$ or $25$ to the previous partitions (and taking into account that $39\vdash[3,9,27]$ is a rational partition), also the cases $n\equiv 7$ and $n\equiv 31$ (mod $32$) are covered. 
                
            \end{enumerate}
            
          \medskip
            \item $n$ is congruent to $8$ or  $9$ mod $32$.
            
          \noindent  Observe that $n=8+32k\vdash[3,9,8k-1, 3\cdot(8k-1)]$ is a rational partition for every positive integer $k$, and this covers the case $n\equiv 8$ (mod $32$). Adding $1$ to this family of partitions we cover the other case (taking also into account that $9\vdash[9]$).
          
\medskip          
 \item $n$ is congruent to $10$, $11$, $19$ or $20$ mod $32$. 
 
 \noindent Let $n\equiv 10$ (mod $32$). We can observe that $42$ is the answer of the question \say{what is the unique number that is $10$ modulo $32$ which does not admit an rational partition?}. Now, we have $10\vdash [1,9],$ $74\vdash [ 25, 49 ]$ and $106\vdash [25,81]$. Clearly, we just need to study the cases $n=74+32k$ and $106+32k$ where $k$ is an odd integer.
              \begin{enumerate}
                  \item Let $n=74+32k$ with $k>1$; then $$n\vdash[25, 49, k, 3k, 7k, 21k]$$ is a rational partition unless $k\in\{7,25,49\}$, and $$n\vdash[25, 49, 3k, 5k, 9k, 15k]$$ covers the remaining cases (note that, for $k=41$, we can choose the latter partition).

                  \item Let $n=106+32k$; then $$n\vdash[25, 81, k, 3k, 7k, 21k]$$ is a rational partition unless $k\in\{25,27,81\}$, and $$n\vdash[25, 81, 3k, 5k, 9k, 15k]$$ covers $k\in\{25,81\}$ (again, for $k=41$, we choose the latter partition); for $k=27$ we pick $970\vdash[ 97, 9\cdot 97 ]$. 
  \end{enumerate}
                   
We already noted that $42$ does not have any rational partition, and the same holds for $43$ and $52$; on the other hand, $51$ does have the rational partition $[1,5,45]$. Observe also that every partition provided for $n\equiv 10$ (mod $32$) does not contain $1$ and $9$, except for $n\in\{10,138, 170, 202,362,394\}$. For $n=10$ we cannot find any rational partition not containing $1$ or $9$, however the values $11$, $19$ and $20$ are excluded by our hypothesis on $n$. As for the other values, we see that $$138\vdash[5,7,13,21,27,65],\quad 170\vdash[17,9\cdot 17], \quad 202\vdash[81,121],$$ $$362\vdash[5,13,15,39,121,169],\quad 394\vdash[169,15\cdot 15]$$ are all rational partitions not containing $1$ and $9$. Thus every $n$ which is congruent to $10$, $11$, $19$ or $20$ mod $32$ has a rational partition except $11$, $19$, $20$, $42$, $43$, $52$. 

\medskip

              
              
\item $n$ is congruent to $12$, $13$, $21$ or $22$ mod $32$.

                \noindent Let $n\equiv 12$ (mod $32$) and  observe that $76$, $108$ and $140$ are covered by $[ 1, 5, 25, 45 ]$, $[ 5, 13, 25, 65 ]$, $[ 3, 7, 39, 7\cdot 13 ] $ respectively. For $n=108$ it will be useful to consider also the partition $[7,11,35,55]$. 
                So, the relevant cases are of the form $n=108+32k$ and $n=140+32k$ with $k$ odd.
                \begin{enumerate}
                    \item Let $n=108+32k$ with $k>1$; then $$n\vdash[7, 11, 35, 55, 3k, 5k, 9k, 15k]$$
                    is a rational partition provided $k\not \in \{ 7,11\}$ and
                    $$n\vdash[ 5, 13, 25, 65, 3k, 5k, 9k, 15k]$$ covers the others partitions.

                    
                    \item Let $n=140+32k$ with $k$ odd; then
                    $$n\vdash[3,7,39,7\cdot 13,3k, 5k, 9k, 15k]$$
                    is a rational partition if $k\not \in \{1,13\}$. 
                    Otherwise, we pick the following rational partitions: $$172\vdash[ 3, 15, 55, 99 ],\quad 556\vdash[21, 121, 7\cdot 27, 15\cdot 15].$$
\end{enumerate}

Observe that every partition presented above does not contain $1$ and $9$, except for $n\in \{76, 204, 236\}$. Considering the alternative partitions $204\vdash[ 5, 13, 5\cdot 13,  121]$ and $236\vdash[ 23, 35, 9\cdot 7, 5\cdot 23 ]$), we can therefore provide a rational partition also for $n\equiv 13$, $n\equiv 21$ and $n\equiv 22$ (mod $32$) except for $77, 85$ and $86$. However we can take 
$$77\vdash[ 3, 5, 7, 27, 35 ],\quad 85\vdash[ 35, 25, 15, 7, 3 ],\quad
86\vdash[ 35, 27, 9, 7, 5, 3 ].$$

\medskip                 
\item $n$ is congruent to $14$, $15$, $23$ or $24$ mod $32$.

\noindent Let $n\equiv 14$ (mod $32$). Observe that $78$ and $110$ are covered by the rational partitions  $[ 3, 75 ]$ and $[ 11, 99 ]$, respectively. Then the relevant cases left are of the form $n=78+32k$ and $n=110+32k$ with $k$ odd.
            \begin{enumerate}
                \item Let $n=78+32k$ with $k>1$ odd. Then we have that
                $$n\vdash [3,75,3k, 5k, 9k, 15k]$$ is a rational partition provided $k\not \in \{ 3, 15,25\}$.  Otherwise consider $$n\vdash [3, 5, 7, 15, 21, 27,3k, 5k, 9k, 15k]$$ for $k\in \{15,25\}$, together with $174\vdash[ 27, 3\cdot 49 ]$.
                \item Let $n=110+32k$ with $k$ odd; then
                $$n\vdash [11,99,3k, 5k, 9k, 15k]$$
                is a rational partition provided $k\not\in \{11, 33\}$. Otherwise we consider  $$n\vdash[3, 5, 7, 15, 35, 45,3k, 5k, 9k, 15k].$$ 
\end{enumerate}                
Observe that $1$ and $9$ do not appear in any partition presented above, except for $n\in \{142, 206\}$. In these cases we pick $$142\vdash[3,15,17,21,35,51],\quad 206\vdash[3,7,11,19,33,7\cdot 19].$$

and, adding $1,9$ or both to every partition presented above, we provide a rational partition for every $n\equiv 15, 23, 24$ (mod $32$) except for $n\in \{23,24,47, 55, 56\}$. In these cases we pick $$23\vdash [3,5,15],\quad 24\vdash[1,3,5,15],\quad 47\vdash [ 33, 11, 3 ], $$ $$55\vdash [ 39, 13, 3 ],\quad 56\vdash [ 39, 13, 3, 1].$$ 

            
\medskip

            \item $n$ is congruent to $16$ or $17$ mod $32$.
            
            \noindent Let $n\equiv 16$ (mod $32$). Observe that $n=16+32k\vdash [3,5,3(1+4k), 5(1+4k)]$ is a rational partition for every $k>0$. Now we can add $1$ to any partition of this kind to get a rational partition for any $n\equiv 17$ (mod $32$) as in our hypothesis.  
\medskip
            \item $n$ is congruent to $18$, $27$ or $28$ mod $32$.
            
\noindent Let $n\equiv 18$ (mod $32$).
Observe that $50\vdash[ 5, 45 ]$, $82\vdash [81, 1]$, $114\vdash[ 3, 5, 7, 13, 21, 65 ]$ and $146\vdash[ 25, 121 ]$. Let us consider $n=114+32k$ and $n=146+32k$ with $k$ odd.
            \begin{enumerate}
                \item Let $n=114+32k$ with $k>1$ odd; then
                $$n=114+32k\vdash [ 3, 5, 7, 13, 21, 65, 3k, 5k, 9k, 15k]$$ is a rational partition provided that $k\not\in \{7,13\}$.
                Otherwise consider
                $$338\vdash[ 13, 13\cdot 25 ],\quad 530\vdash[ 53, 9\cdot 53].$$
                \item Let $n=146+32k$ and consider
                $$n\vdash [ 25, 121, 3k, 5k, 9k, 15k]$$
                is a rational partition unless $k=5$, for which we pick $306\vdash [ 25, 121, 5, 15, 35, 21\cdot 5].$
                \end{enumerate}
                
We can observe that $1$ and $9$ do not appear in any partition presented above, except for $n\in\{82, 178, 210, 242\}$. If $n\neq 82$, we can consider the alternative partitions $$178\vdash[7,11,15,27,55,63],\quad 210\vdash[21,27\cdot 7],\quad 242\vdash[7,11,21,33,49,121];$$ Adding $9$ or concatenating $[1,9]$ we get the desired conclusion for every $n\equiv 27$ or $n\equiv 28$ (mod $32$), with the exception of $n=92$. For this, we take $92\vdash[5,9,13,65]$.
\end{itemize}
We considered all the possible cases for $n$, so the ``if part" of the proof is complete. The converse statement is just a check.
\end{proof}

We are now ready to prove Theorem~A. It will be convenient to treat separately the case when $n$ is \say{small} (Lemma~\ref{small}), and to introduce some further notation.

Given a partition $\mu=[\mu_1, \dots, \mu_{l_\mu}]$ of a positive integer $n$, and an integer $r$ coprime with each of the $\mu_i$, we set $\tau^{(r)}_{\mu}=\sum_{i=1}^{l_\mu} \tau_{\mu_i}^{(r)}$ (where the $\tau_{\mu_i}^{(r)}$ are as in Lemma \ref{lemma2}). 

Also, for a positive integer $n$ and an odd prime $p$, we denote by $u_{p,n}$ an integer that is coprime with $n$, and that satisfies the following properties: $u_{p,n}\equiv 1$ (mod $n_{p'}$), and $o_n(u_{p,n})=\varphi(n_p)$.  
Moreover, let us denote with $v_{2,n}$ an integer that is coprime with $n$, congruent to $-1$ modulo $n_2$ and congruent to $1$ modulo $n_{2'}$; finally, denote by $u_{2,n}$ an integer that is coprime with $n$, congruent to $3$ modulo $n_2$ and congruent to $1$ modulo $n_{2'}$. Whenever there is no ambiguity on $n$, we simply adopt the notation $u_p,u_2,v_2$.

In \cite{PV24} we have seen that, to any uniformly semi-rational group we can associate a group $\mathcal{R}_G$ and a coset $\mathcal{S}_G$ called rationality and semi-rationality of $G$; those are a way to measure how the group is far from being rational. For the sake of completeness, we recall next the relevant definitions.

\begin{definition}\label{rationality}    
Let $G$ be a group, and let $n$ denote the exponent of $G$. Denoting by $x^G$ the conjugacy class of $x\in G$, define $$\rg G=\{j\in \mathcal{U}(\Z/n\Z)\mid x^j\in x^G\; {\textnormal{ for every }} x\in G\};$$ clearly $\rg G$ is a subgroup of 
$\mathcal{U}(\Z/n\Z)$, which we call the \emph{rationality} of $G$. Now, if $r$ is an integer such that $G$ is $r$-semi-rational, then it is easy to see that the set $$\sg G=\{s\in \mathcal{U}(\Z/n\Z)\mid {\text{ $G$ is $s$-semi-rational}}\}$$ coincides with the coset of $r$ modulo $\rg G$ (in both the definitions of $\rg G$ and $\sg G$ we slightly abuse the notation identifying an integer with its congruence class modulo $n$). In this case, we refer to this coset as to the \emph{semi-rationality} of $G$.
\end{definition} 

\begin{lemma}\label{small}
Let $n$ be a positive integer such that $n< 110$. The alternating group $A_n$ is uniformly semi-rational if and only if $n<23$ and $n\not\in \{16, 21\}$.
\end{lemma}
\begin{proof}
We begin proving that $A_{16}$ and $A_{21}$ are not uniformly semi-rational groups.
To achieve this result we will provide a set $\Lambda=\{\lambda_1, \dots, \lambda_m, \lambda_{m+1}\}$ of odd-distinct non-rational partitions of $n$ such that, for every $r$ coprime with the exponent of $A_{n}$, there exists at least one $\lambda\in\Lambda$ with $\tau^{(r)}_{\lambda}=0$. Our strategy will be to construct a set as above such that if $r$ satisfies $\tau^{(r)}_{\lambda_i}=1$ for every $i\in\{1,\dots, m\}$, then $\tau^{(r)}_{\lambda_{m+1}}=0$. 
In other words, for any choice of $r$, there exists a 
partition $\lambda\in\Lambda$ such that 
any element $x\in A_n$ of type $\lambda$ lies in a conjugacy class of $S_n$ that splits into two conjugacy classes of $A_n$, is non-rational and $x^r\in x^{A_n}$ (hence, $x$ is not $r$-semi-rational). 
For the sake of clarity, whenever $\mu$ is a partition of $n$ we will write $\tau_\mu$ omitting the choice of $r$, since it will not affect the argument. Now, 
if $n=16$ then we can consider the set $\{[9,7], [15,1], [7,5,3,1]\}$, otherwise if $n=21$ then we can choose $\{[21], [11,9,1],[11,7,3]\}$.  

For all the remaining values of $n<23$, the alternating group $A_n$ is uniformly semi-rational; this can be checked directly for $n\leq 19$ using the character tables available in $\texttt{GAP}$, thus we present  a proof for the cases $n=20,22$.

 We start reckoning the rationality of the group and we observe that this is equivalent to impose $\tau_\lambda^{(r)}=0$ for any partition $\lambda$ of $n$. 
In both cases $n=20$ and $n=22$ this yields $\tau_p^{(r)}=0$ for every odd prime $p$ dividing the exponent of $A_n$: we claim this in turn implies $\mathcal{R}_{A_n}= \la u_2, v_2, u_p^2 |\; p\neq 2\text{ divides }exp(A_n) \ra$. In fact, since $\tau_p^{(r)}=0$ (for every prime $p$ as above) if and only if $o_p(r)_2$ is not equal to $(p-1)_2$, we have that the congruence class of $r$ modulo $p$ lies in the cyclic subgroup generated by the congruence class of $u_p^2$ modulo $p$. 

Now, to determine ${\mathcal{S}}_{A_{20}}$, we need to find one integer $r$ such that $A_{20}$ is $r$-semi-rational, that translates to $\tau_\lambda^{(r)}=1$ for any odd-distinct non-rational partition $\lambda$ of $20$. It can be checked that such an $r$ is $u_3\cdot u_5\cdot u_{11}\cdot u_{13}\cdot u_{19}$. Similarly, for $n=22$, $r$ can be chosen as $u_7\cdot u_{13}\cdot u_{17}\cdot u_{19} $.

Finally, consider $23\leq n\leq 68+41$.
Then we have only finitely many cases to check, and  
 we can use $\texttt{GAP}$ to conclude that $A_n$ is not uniformly semi-rational for any $n$ as above. 
\end{proof}







\begin{ThmA}
   The alternating group $A_n$ is a uniformly semi-rational group if and only if $n<23$ and $n\not\in\{16, 21\}$.
\end{ThmA}
\begin{proof}
The \say{if} part of this theorem is proved by the previous lemma. As for the \say{only if} part, again in view of the same lemma, it remains to prove that $A_{n+41}$ is not uniformly semi-rational assuming $n\geq 69$.
Now, let us consider the following set of integers:
$$T=\{7d^2, 11d^2, 15d^2, 41d^2, 55d^2, 77d^2| \;d \text{ is odd} \},$$
Pick a positive integer $n\not \in T$ with $n\geq 69$. We claim that $A_{n+41}$ is not uniformly semi-rational. 
To achieve this result we will use the technique explained in the first paragraph of the proof of Lemma~\ref{small}, i.e. assuming the contrary, we will provide a set $\Lambda$ of partitions of $n+41$ that leads to a contradiction.  Again, for the sake of clarity, whenever $\mu$ is a partition of $n+41$ we will write $\tau_\mu$ omitting the choice of $r$, since it will not affect the argument.


Let us begin considering $n$ odd and take the partition $\lambda_1=[n,41]$, since $n\neq 41d^2$ then $\lambda_1$ is not rational. Since $n\geq 69$, by Theorem~\ref{partitions} there exists a rational partition $\mu=[\mu_1, \dots, \mu_{l_\mu}]$ of $n$.
Recall that, as noted in the paragraph preceding Theorem~\ref{partitions}, we can choose $\mu$ not containing the number $41$. Now consider $\lambda_2\vdash [\mu_1, \dots, \mu_{l_\mu}, 41]$, thus assuming $\tau_{\lambda_2}=1$ implies that $\tau_{41}=1$ and, since $\tau_{\lambda_1}=1$, we also have that $\tau_n=0$. We display next a set  of partitions which, together with $\lambda_1$ and $\lambda_2$, will constitute a set $\Lambda$ as described before. Consider 
$$\lambda_3=[n, 33,5,3]$$
 and 
$$\lambda_4=[n,27,11,3].$$
Observe that $\lambda_3$ and $\lambda_4$ are (odd-distinct) non-rational partitions, otherwise $n$ would be of the form $55d^2$ or  $11d^2$ contradicting $n\not\in T$. Thus, $$\tau_{\lambda_3}=\tau_n+\tau_{33}+\tau_5+\tau_3=0+2\tau_{3}+\tau_{11}+\tau_5=\tau_{11}+\tau_5=1$$ and $\tau_{\lambda_4}=\tau_{27}+\tau_{11}+\tau_3=4\tau_3+\tau_{11}=\tau_{11}=1$. As a consequence, we get $\tau_5=0$. Consider now $$\lambda_5=[n,27,9,5]$$
 and 
$$\lambda_6=[n, 21,15,5]$$ As above, $\lambda_5$ and $\lambda_6$ are odd-distinct and non-rational. Then $\tau_{\lambda_5}=\tau_{27}+\tau_9+\tau_5=\tau_3+\tau_5=\tau_3=1$ and  $\tau_{\lambda_6}=\tau_{21}+\tau_{15}+\tau_5=2\tau_3+2\tau_5+\tau_7=\tau_7=1$. Finally, consider the odd-distinct and non-rational partition 
$$\lambda_7=[n, 15, 11, 7, 5, 3]:$$ then $\tau_{\lambda_7}=2\tau_3+ 2\tau_5 + \tau_{11}+ \tau_7=0$, and this proves our claim for this case. 

Let us now suppose that $n\geq 70$ is even and $n-1\not \in T$. 
Now we can use the same argument as before, replacing $n$ with $n-1$ and adding $1$ to every partition; the only exception is the partition $\lambda_2$, because in principle $1$ can appear in $\mu$, but in this case we can just replace $\mu$ with a rational partition of $n$.

Finally, we conclude by assuming  $n\in T$.

If $n$ is of the form $7d^2$, we can still consider $\lambda_1, \lambda_3, \lambda_4, \lambda_5$. Note that we cannot use the partition $\lambda_6$ since, in this case, it is rational. Let us begin with $\tau_{41}=1$, so that $\tau_7=0$. 
Assuming $\tau_{\lambda_3}=\tau_{\lambda_4}=1$, as above we get $\tau_{11}=1, \; \tau_5=0$ and, if $\tau_{\lambda_5}=1$ then $\tau_3=1$. At this point, defining $\lambda'_6=[7d^2, 21, 11, 9]$, we obtain $\tau_{\lambda'_6}=0$. We can use the same set of partitions to prove our claim also when  $\tau_{41}=0$, which implies $\tau_7=1$. In fact, assuming $\tau_{\lambda_3}=\tau_{\lambda_4}=1$ we get $\tau_5=\tau_{11}=0$, and if $\tau_{\lambda_5}=1$ then $\tau_3=0$. Again, we have $\tau_{\lambda'_6}=0$.

If $n=11d^2$ then $\lambda_3, \lambda_5, \lambda_6,\lambda_7$ are non-rational partitions. Let us start assuming $\tau_{11}=1$; in this case $\tau_{\lambda_6}=1$ implies $\tau_7=0$, but then we get $\tau_{\lambda_7}=0$. On the other hand, let  $\tau_{11}=0$: then, assuming $\tau_{\lambda_3}=1, \; \tau_{\lambda_6}=1$ and $\tau_{\lambda_5}=1$, we get $\tau_5=1, \; \tau_{7}=1$ and $\tau_3=0$. Thus, considering the odd-distinct non-rational partition $\rho=[11d^2, 25, 7, 5, 3,1]$, we get $\tau_{\rho}=0$.

If $n=15d^2$ then $\lambda_3, \lambda_4$ are odd-distinct non-rational partitions. Assume first  $\tau_{15}=0$: in this case $\tau_{\lambda_4}=1$ implies $\tau_{11}=1$, and $\tau_{\lambda_3}=1$ yields $\tau_5=0$. But since $\tau_{15}=0$ we have in fact also $\tau_3=0$. Now, setting $\rho_1=[15d^2, 25, 9, 7], \; \rho_2=[15d^2, 21, 11, 5, 3, 1]$, these are odd-distinct non-rational partitions. If $\tau_{\rho_1}=1$, then $\tau_7=1$ and we get $\tau_{\rho_2}=0$. On the other hand, let $\tau_{15}=1$. Assuming $\tau_{\lambda_4}=1$ and $\tau_{\rho_1}=1$, we get $\tau_{11}=0$ and $\tau_{7}=0$; furthermore, assuming $\tau_{\lambda_3}=1$, we get $\tau_3=1$. But since $\tau_{15}=1$, we also have $\tau_{5}=0$. Finally, considering the odd-distinct non-rational partition $\rho_3=[15d^2, 33, 7, 1]$, we get $\tau_{\rho_3}=0$.

For $n=41d^2$, consider the partitions $\lambda_2=[d^2, 3d^2, 7d^2, 9d^2, 21d^2, 41]$, $\lambda_4, \lambda_6$, which are odd-distinct and non-rational. Assuming $\tau_{\lambda_2}=1$ we obtain $\tau_{41}=1$. Moreover, assuming $\tau_{\lambda_4}=1$ and $\tau_{\lambda_6}=1$, we get $\tau_7=0$ and $\tau_{11}=0$. Finally, considering the odd-distinct non-rational partition $\rho=[3d^2, 11d^2, 27d^2,21,15,5]$, we see that $\tau_\rho=0$. 

As the last cases for $n\in T$, let us assume $n=md^2$ for $m\in \{ 55, 77\}$ and consider the odd-distinct non-rational partitions $\lambda_1$, $\lambda_4$, $\lambda_6$. By Theorem~\ref{partitions}, $m$ has a rational partition $\mu=[\mu_1,\dots,\mu_{l_\mu}]$. Let us consider also $\lambda_2=[\mu_1 d^2, \dots, \mu_{l_\mu}d^2, 41]$: assuming $\tau_{\lambda_2}=1$ we get $\tau_{41}=1$, and assuming $\tau_{\lambda_1}=1$ we get  $\tau_m=0$. Therefore, always assuming $\tau_{\lambda_4}=1$ and $\tau_{\lambda_6}=1$,  we have that $\tau_{7}=1$ and $\tau_{11}=1$.  Finally, let us consider the odd-distinct non-rational partition $\mu=[md^2, 15, 11, 7, 5, 3]$ and observe that $\tau_{\mu}=0$.
This concludes the proof.\end{proof}

For every uniformly semi-rational alternating group, in the second column of Table~\ref{Prima} and Table~\ref{Seconda} we provide one choice of an integer $r$ such that the corresponding group is $r$-semi-rational, and in the third column we display the semi-rationality of the group (this information is obtained by direct computation from the character tables available in \texttt{GAP}). Note that the semi-rationality is given in the form $r\rg G$, i.e. as the coset of $r$ modulo the rationality of the group.
The groups are placed in the two tables according to one particular invariant that we introduce next. 

\begin{definition}
Let $G$ be a uniformly semi-rational group. Then $G$ is called $2^k$-USR for some positive integer $k$ if $2^k$ is the largest power of $2$ that divides the order of any $r\in \mathcal{S}_G$ modulo the exponent of the group. 
\end{definition}

Note that cut groups are $2$-USR, so alternating cut groups are displayed in the first six rows of Table~\ref{Prima}. As mentioned in the Introduction, the classification of these groups is already known (see~\cite{Ferraz}); the new information here is their semi-rationality.

\begin{table}[ht]
\caption{$2$-USR alternating groups.}
\begin{center}

\small \begin{tabular}{|c|c|c|}
\hline
$G$ & $r$ &$\mathcal{S}_{G}$\\
\hline

$A_3$ & $-1$ &  $-1\cdot \la 1\ra$\\
$A_4$ & $-1$ &$-1\cdot \la 1\ra$\\

$A_7$ & $-1$ &$-1\cdot \la v_2, u_3, u_5, u_7^2\ra$\\
$A_8$ &  $-1$ &$-1\cdot \la v_2, u_3\cdot u_5, u_3^2, u_5^2, u_7^2  \ra$\\
$A_9$ & $-1$ & $-1\cdot \la v_2, u_3\cdot u_5, u_7, u_3^2, u_5^2  \ra$\\
$A_{12}$ & $-1$ &$-1\cdot \la u_2, v_2, u_5\cdot u_7, u_3^2,  u_5^2, u_7^2, u_{11}^2 \ra$\\
\hline
$A_{10}$ & $11$ & $11\cdot \la u_2, v_2, u_5, u_3\cdot u_7, u_3^2, u_7^2, u_{11}^2  \ra$\\
$A_{11}$ & $13$ &$13\cdot \la u_2, v_2, u_5, u_3\cdot u_7, u_3^2, u_7^2, u_{11}^2  \ra$\\ 
\hline

\end{tabular}
\end{center}
\label{Prima}
\end{table}

{\small
\begin{table}
\caption{$4$-USR and $16$-USR alternating groups.}
\begin{center}
\begin{tabular}{|c|c|c|}
\hline
$G$ & $r$ &$\sg G$\\

\hline

$A_5$ & $13$ & $13\cdot \la u_3, u_5^2 \ra$\\ 
$A_6$ & $13$ & $13\cdot\la v_2, u_3, u_5^2  \ra$\\ 
$A_{13}$ & $41$ & $41\cdot \la u_2, v_2, u_{11}, u_5\cdot u_7, u_3^2, u_5^2, u_7^2, u_{13}^2 \ra$\\ 
$A_{14}$ & $47$ & $47\cdot \la u_2, v_2, u_7, u_3\cdot u_{11}, u_3^2, u_5^2, u_{11}^2, u_{13}^2 \ra$\\
$A_{15}$ & $43$ & $43\cdot \la u_2, v_2, u_{13}, u_3^2, u_5^2, u_7^2, u_{11}^2   \ra$\\
$A_{19}$ & $103$ & $103\cdot \la u_2, v_2, u_3^2, u_{17},u_3\cdot u_7\cdot u_{11},   u_5^2, u_7^2, u_{11}^2, u_{13}^2, u_{19}^2   \ra$\\
$A_{20}$ & $18707041$ & $18707041\cdot \la u_2, v_2, u_3^2, u_5^2, u_7^2, u_{11}^2, u_{13}^2, u_{17}^2, u_{19}^2\ra$\\
\hline
\hline
$A_{17}$ & $97$ & $97\cdot \la u_2, v_2, u_3\cdot u_5\cdot u_{11}\cdot u_{13}, u_3^2, u_5^2, u_7^2, u_{11}^2, u_{13}^2, u_{17}^2\ra$\\
$A_{18}$ & $277$ & $277\cdot \la u_2, v_2, u_7\cdot u_{11}, u_3^2, u_5^2, u_7^2, u_{11}^2, u_{13}^2, u_{17}^2 \ra$\\
$A_{22}$ & $16505777$ & $16505777\cdot \la u_2, v_2, u_3^2, u_5^2, u_7^2, u_{11}^2, u_{13}^2, u_{17}^2, u_{19}^2\ra$\\
\hline

\end{tabular}
\end{center}
\label{Seconda}
\end{table}
\vspace{2cm}
}
\newpage

\section{Non-alternating simple groups}
In this section, we classify the remaining uniformly semi-rational non-abelian simple groups.  We can use the classification of semi-rational simple groups presented in \parencite{Alavi2017}, keeping in mind that $^3D_4(2), ^3D_4(3), ^2B_2(8), ^2B_2(32), ^2G_2(27)$ and the Tits group $^2F_4(2)'$ are not supposed to be there and the group $G_2(4)$ is missing, as pointed out in \parencite[Theorem 5.1]{bacjes}. 

\begin{ThmC}\label{thm2}
Let $G$ be a non-abelian simple group which is not an alternating group. Then $G$ is uniformly semi-rational if and only if it is isomorphic to one of the groups in Table~\ref{sporadic1}, Table~\ref{sporadic2}, Table~\ref{lie1} and Table~\ref{lie2}.
\end{ThmC}
  
  \begin{proof}
Since uniformly semi-rational groups are semi-rational, by \cite{Alavi2017} and \cite{bacjes} we have to deal with finitely many groups. Using the character tables available in $\texttt{GAP}$ it turns out that, in fact, each of these groups is uniformly semi-rational, and it is also possible to determine their semi-rationality.  
  \end{proof}
  
Among the groups classified in the above theorem, the cut groups appear in Table~\ref{sporadic1} and in Table~\ref{lie1} (see also \cite{bacjes}) with their semi-rationalities.
  
\begin{rem}
By the previous theorem, we can observe that every non-abelian non-alternating semi-rational simple group is uniformly semi-rational. Similarly, it can be checked that also the quadratic rational non-abelian non-alternating simple groups (classified in \cite{CutTrefethen2017}) are all semi-rational (hence uniformly semi-rational).
\end{rem}

\begin{table}[H]
\caption{$2$-USR sporadic simple groups.}

\begin{center}
	\begin{tabular}{|c|c|c|}
	\hline
	$G$ 	& $r$	& $\sg G$\\
	\hline
	$M_{11}$ & $-1$	& $-1\cdot\la u_2, u_3, u_5, u_{11}^2\ra$\\
	$M_{12}$ & $-1$	&$-1\cdot\la u_2, v_2, u_3, u_5, u_{11}^2\ra$\\
	$M_{22}$ & $-1$	&$-1\cdot\la u_2, v_2, u_3, u_5, u_7^2, u_{11}^2\ra$\\ 
	$M_{23}$ & $-1$	&$-1\cdot\la u_2, v_2, u_5^2, u_3\cdot u_5, u_7^2, u_{11}^2, u_{23}^2 \ra$\\
	$M_{24}$ & $-1$	&$-1\cdot\la  u_2, v_2, u_5^2, u_3\cdot u_5, u_7^2, u_{11}, u_{23}^2 \ra$\\
	$McL$ 	& $-1$	&$-1\cdot\la u_2, v_2, u_3^2, u_5^2, u_7^2, u_{11}^2 \ra$\\
	$Co_1$ & $-1$	&$-1\cdot\la  u_2, v_2, u_3^2, u_5, u_7, u_{11}, u_{13}^2, u_3\cdot u_{13}^2, u_{23}^2 \ra$ \\ 
	$Co_2$ & $-1$	&$-1\cdot \la u_2, v_2, u_3^2, u_5^2, u_3\cdot u_5, u_7^2, u_{11}, u_{23}^2 \ra$ \\ 
	$Co_3$ & $-1$	&$-1\cdot \la u_2\cdot v_2, u_3, u_5^2, v_2\cdot u_5, u_7, u_{11}^2, u_{23}^2 \ra$ \\ 
	$Th$ 	& $-1$	&$-1 \cdot \la v_2, u_7, u_{19}, u_3^2, u_5^2, u_{13}^2, u_{31}^2 \ra$ \\ 
	$M$ 	& $-1$	&$-1 \cdot \la u_{41}, u_2\cdot u_5\cdot u_7\cdot u_{17}\cdot u_{19}, u_2\cdot u_5\cdot u_7\cdot u_{17}\cdot u_{19},$\\
		&	&$ u_2\cdot v_2\cdot u_3\cdot u_7\cdot u_{13}\cdot u_{17}\cdot u_{29}, v_2\cdot u_3\cdot u_5\cdot u_{13}\cdot u_{19}\cdot u_{29},$ \\
		&	&$ u_2^2, u_3^2, u_5^2, u_7^2, u_{11}^2, u_{13}^2, u_{17}^2, u_{19}^2, u_{23}^2, u_{29}^2, u_{31}^2, u_{47}^2, u_{59}^2, u_{71}^2 \ra$ \\ 
	\hline
	$HS$ 	& $31$	&$31\cdot\la u_2\cdot v_2, u_3, u_5^2, u_2\cdot u_5, u_7, u_{11}^2 \ra$\\ 
	\hline
	
\end{tabular}
\end{center}
\label{sporadic1}
\end{table}

\begin{table}[H]
\caption{$4$-USR and $16$-USR sporadic simple groups.}

\begin{center}
	\begin{tabular}{|c|c|c|}
	\hline
	$G$ & $r$ &$\sg G$\\
	
	\hline
	$J_2$ 	& $13$	&$13\cdot \la u_2, v_2, u_3, u_5^2, u_7\ra$\\ 
	$Suz$ 	& $47$	&$47\cdot\la u_2, v_2, u_3^2, u_5^2, u_7^2, u_{11}, u_{13}^2\ra$\\ 
	$Fi_{22}$ & $-67$ 	&$-67\cdot \la u_2, u_3^2,u_5, u_7, u_{11}^2, u_{13}^2,  \ra$ \\ 
	$Fi_{23}$ & $-59$	&$-59\cdot \la u_2, u_3, u_5, u_7, u_{11}^2, u_{13}^2, u_{23}^2\ra$  \\ 
	$Fi'_{24}$ & $137$ 	&$137\cdot \la u_2\cdot v_2, u_{17}, u_3^2, u_5^2, u_7^2, u_{11}^2, u_{13}^2, u_{23}^2, u_{29}^2  \ra$ \\ 
	$HN$ 	& $67$	&$67\cdot \la u_2, u_3, u_{11}, u_5^2, u_7^2, u_{19}^2  \ra$ \\ 
	
	\hline
	\hline
	
	$He$ 	& $41$	&$41\cdot \la u_2, v_2, u_3^2, u_5, u_7^2, u_{17}^2 \ra$\\ 
	$B$ 	& $-683$	&$-683\cdot \la u_{11}, u_{13}, u_{19}, u_3\cdot u_5,  u_2^2, u_3^2, u_5^2, u_7^2, u_{17}^2, u_{23}^2, u_{31}^2, u_{47}^2 \ra$ \\ 
	
	\hline
	
\end{tabular}
\end{center}
\label{sporadic2}
\end{table}

\begin{table}[ht]
\caption{Rational and $2$-USR simple groups of Lie type.}
\begin{center}
	\begin{tabular}{|c|c|c|}	
	\hline
	
	$G$ & $r$ &$\sg G$\\
	
	\hline
	
	$Sp(6,2)$ &$1$	&$\la u_2, v_2,  u_3, u_5, u_7\ra$\\
	$O^+_{8}(2)$ &$1$	&$\la u_2, v_2, u_3, u_5, u_7 \ra$\\ 
	
	\hline
	\hline
	$L_{2}(7)$ & $-1$	&$-1\cdot\la v_2, u_3, u_7^2 \ra$ \\ 
	$U_{3}(3)$ & $-1$	&$-1\cdot\la u_2\cdot v_2, u_3, u_7^2 \ra$ \\ 
	$U_{3}(5)$ & $-1$	&$-1\cdot\la u_2, u_3, u_5, u_7^2 \ra$ \\ 
	$U_{4}(3)$ & $-1$	&$-1\cdot\la u_2, u_5, u_3^2, u_7^2 \ra$ \\ 
	$U_{5}(2)$ & $-1$	&$-1\cdot\la u_2, v_2, u_5, u_3^2, u_{11}^2 \ra$ \\ 
	$U_{6}(2)$ & $-1$	&$-1\cdot\la u_2, v_2, u_5, u_7, u_3^2, u_{11}^2 \ra$ \\ 
	
	\hline
	
	$Sp(4,3)$ & $11$	&$11\cdot\la u_2, v_2, u_5, u_3^2 \ra$ \\ 
	\hline
\end{tabular}
\end{center}
\label{lie1}
\end{table}

\begin{table}[ht]
\caption{$4$-USR and $16$-USR simple groups of Lie type.}
\begin{center}
	\begin{tabular}{|c|c|c|}	
	\hline
	
	$G$ &$r$ &$\sg G$\\
	
	\hline
	
	$L_{2}(11)$ &$7$ 	&$7\cdot\la u_3, u_5^2, u_{11}^2 \ra$ \\ 
	$L_{3}(4)$ 	& $13$ &$13\cdot\la v_2, u_3, u_5^2, u_7^2 \ra$ \\ 
	
	$O_{7}(3)$ 	& $11$ &$11\cdot\la u_2, v_2, u_5, u_7, u_3^2, u_{13}^2 \ra$ \\ 
	$O^+_{8}(3)$ & $11$	&$11\cdot\la u_2, v_2, u_3, u_5, u_7, u_{13}^2 \ra$ \\ 

	$G_{2}(3)$ & $5$ 	&$5\cdot\la u_2, v_2, u_7, u_3^2, u_{13}^2,  \ra$ \\ 
	$G_{2}(4)$ & $137$ 	&$137\cdot\la u_2, v_2, u_3^2, u_5^2, u_7^2, u_{13}^2 \ra$ \\ 
	$Sp(6,3)$ & $11$	&$11\cdot\la u_2\cdot v_2, u_5, v_2\cdot u_7, u_3^2, u_7^2, u_{13}^2 \ra$ \\ 
	
	\hline
	\hline
	
	$^2E_6(2)$ & $41$	&$41\cdot\la u_2, v_2, u_{13}, u_5\cdot u_7, u_3^2, u_5^2, u_7^2, u_{11}^2, u_{17}^2, u_{19}^2 \ra$ \\ 
	$F_{4}(2)$ & $11$	&$11\cdot\la u_2\cdot v_2, u_5, u_7, u_{13}, u_3^2, u_{17}^2  \ra$ \\ 
	$Sp(8,2)$ & $11$	&$11\cdot\la u_2, v_2, u_3, u_5, u_7, u_{17}^2 \ra$ \\ 

	\hline
	\end{tabular}
\end{center}
\label{lie2}
\end{table}

\newpage
\smallskip
{\bf Acknowledgment.} The author would like to express his deep gratitude to Emanuele Pacifici for his crucial help and careful reading of this manuscript.  
A special thanks goes to Greta Mori for her computational support.


\printbibliography

\end{document}